\documentclass[11pt]{article}
\usepackage{format}
\usepackage[english]{babel}
\usepackage[utf8]{inputenc}

\usepackage[sort, numbers]{natbib}
 
\usepackage{preamble}

\newcommand{\new}[1]{}
\renewcommand{\new}[1]{#1}

\title{\vspace{-1in}\huge
Second-order bounds for the M/M/$s$ queue with random arrival rate}

\author{
\normalsize Wouter J.E.C. van Eekelen \vspace{-1mm}\\ 
\scriptsize Booth School of Business, University of Chicago, \href{w.j.e.c.vaneekelen@tilburguniversity.edu}{Wouter.vanEekelen@chicagobooth.edu}   \vspace{.5mm} \\ 
\normalsize Grani A. Hanasusanto \vspace{-1mm}\\ 
\scriptsize Industrial and Enterprise Systems Engineering, University of Illinois Urbana-Champaign, \href{gah@illinois.edu}{gah@illinois.edu}   \vspace{.5mm} \\ 
\normalsize John J. Hasenbein \vspace{-1mm}\\ 
\scriptsize Operations Research and Industrial Engineering, The University of Texas at Austin, \href{jhas@mail.utexas.edu}{jhas@mail.utexas.edu} \vspace{.5mm} \\
\normalsize Johan S.H. van Leeuwaarden \vspace{-1mm}\\ 
\scriptsize Department of Econometrics and Operations Research, Tilburg University,  \href{j.s.h.vanleeuwaarden@tilburguniversity.edu}{j.s.h.vanleeuwaarden@tilburguniversity.edu}  \vspace{.5mm} \\
} 

\date{\small\mydate{\today}}

\linedabstractkw{
Consider an M/M/$s$ queue with the additional feature that the arrival rate is a random variable of which only the mean, variance, and range are known. Using semi-infinite linear programming and duality theory for moment problems, we establish for this setting tight bounds for the expected waiting time. These bounds correspond to an arrival rate that takes only two values. The proofs crucially depend on the fact that the expected {waiting} time, as function of the arrival rate, has a convex derivative. We apply the novel tight bounds to a rational queueing model, where arriving individuals decide to join or balk based on expected utility and only have partial knowledge about the market size. 
}{Second-order bounds, rational queueing, M/M/$s$ queue, Poisson mixture model, parametric uncertainty} 

\begin{document}

\maketitle

\section{Introduction}
The majority of stochastic models in queueing theory assume known arrival processes, which facilitates exact analysis. In fact, the dominant assumption is that potential customers arrive according to a Poisson process with a known intensity $\lambda$. In contrast, we interpret the arrival rate as an unknown parameter of which partial information is available. 
The parameter $\lambda$ is then replaced by a random variable $\Lambda$ with some distribution representing what is believed or known about the market size. We primarily focus on the situation where we know the mean and variance of $\Lambda$, with the connection to classical queueing theory when the variance is zero and, hence, the random variable $\Lambda$ is known to be $\lambda$ always. 

{We set out to solve an optimization problem for determining tight bounds for the expected waiting time for all distributions of $\Lambda$ that comply with the partial information. }
For mean-variance information, we will show that these tight bounds are attained by two-point distributions, essentially saying that the worst-case market is one where 
the actual market size attains its maximum size or is well below the expected market size. To establish such tight bounds, we invoke primal-dual techniques for solving semi-infinite linear programs. Successfully applying these techniques will prove to depend crucially on the properties of the expected {waiting} time viewed as function of the arrival rate $\lambda$. For the main queueing system in this paper, the M/M/$s$ queue,  we will leverage that the expected {waiting} time is (i) increasing and (ii) convex in $\lambda$, with (iii) the first derivative with respect to (w.r.t) $\lambda$ being a convex function of $\lambda$. All three properties will turn out to be crucial for solving the optimization problems in this paper, and the convex derivative will be particularly useful for the setting with mean-variance information. 

{
We summarize the main contributions in this paper as follows.
For the M/M/$s$ queue with partially known arrival rate, we establish novel tight bounds. When mean-variance-support information is known, we show that the tight bound for the expected \new{waiting} time is attained by a two-point distribution. 
Our proof of these tight bounds combines two results from different areas. The first result stems from the area of optimization and says the semi-infinite linear program that describes the tight bound for the expectation of a function with a convex derivative is attained by a two-point distribution. The second result stems from queueing theory and shows that the expected \new{waiting} time as a function of the arrival rate is convex, and we establish that its derivative is also a convex function of the arrival rate. Combining these two results proves to be an effective strategy for finding distributionally robust bounds for queueing systems subject to parametric uncertainty. 
The tight bounds are leveraged for analysis of unobservable M/M/$s$ queues that cater to rational users who decide to join or balk based on expected utility. We use the wait bounds to bound the equilibrium arrival rate, which in turn, leads to tractable maximin analyses for setting the revenue-maximizing price. In this way, we extend the classical strategic queueing literature with known model parameters as in 
\cite{naor1969regulation,edelson1975congestion} to a setting where the arrival rate is only partially characterized in terms of range, mean, and variance. }




The paper is structured as follows. In Section~\ref{subsec:liter}, we present a brief overview of related literature. Section~\ref{MMSsec2} contains the proof of our main result, which provides bounds for the expected \new{waiting} time under uncertain arrival rates. Section~\ref{MMSsec3} applies these bounds to a rational queueing model. 
Finally, in Section~\ref{MMSsec6}, we conclude and propose \new{some further} research directions.

\subsection{Related literature}\label{subsec:liter}
This research connects four themes in the queueing literature: random arrival rates, second-order bounds, parametric convexity, and rational queueing. 

{\it Queues with random arrival rate}. 
Due to its mathematical tractability, the M/M/$s$ queue is among the cornerstones of queueing theory and is widely applied in service operations management. 
To enhance practical relevance, several studies proposed to relax the assumption of Poisson arrivals, and instead model the arrival rate as a random variable
\citep{avramidis2004modeling,harrison2005method,whitt2006staffing,steckley2009forecast,heemskerk2017scaling,chen2017staffing,zan2018staffing,chen2001two}.
This gives rise to a Poisson mixture model for the arrival process, which can deal with forecast errors, overdispersion (compared with the natural fluctuations of a Poisson process), and unknown market size. The mathematical formulae of the M/M/$s$ model, such as the Erlang-C formula for the delay probability, then still apply once the mixing distribution can be characterized. \citet{jongbloed2001managing} showed how to estimate this mixing distribution based on available data of arrivals. 

{\it  Second-order bounds}. 
For this M/M/$s$ queue with random arrival rate, we derive performance bounds that do not require the full mixture distribution, but only utilize the first two moments of the arrival rate. \new{To derive such second-order bounds,} we need to solve a semi-infinite linear program.  To do so, we first show that the expected stationary waiting time can be written as the expectation of a function of the random arrival rate that has a convex derivative.  We then establish a general result, which provides tight second-order bounds as the solutions to semi-infinite linear programs for the expectation of functions with convex derivatives.
Such semi-infinite linear programs arise naturally in studying second-order bounds, or more generally, moment problems. For example, \citep{birge1991bounding,delage2010distributionally,dokov2005second} use second-moment information to compute tight bounds for distributionally robust stochastic programming. {Second-order bounds also play a predominant role in the theory on performance bounds for queues, which started with the work of \citet{kingman1962some}. Kingman derived bounds for the mean waiting time in the GI/G/1 queue, expressed in terms of the first two moments of the interarrival- and service-time distributions; however, these bounds are not in general tight.}

{\it Convexity and beyond.}
Numerous studies have taken advantage of the insight that moments of the waiting time can be viewed as the expected value of a convex function. \citet{vasicek1977inequality} showed that the variance of waiting time under a general queueing discipline does not exceed that under the LCFS discipline, and \citet{hajek1983proof} established that among all arrival processes for an exponential server queue with specified arrival and service rates, the arrival process which minimizes the expected \new{waiting} time is the process with constant interarrival times. \citet{weber1980note} revealed that the mean queueing time in the G/GI/$m$ queue is nonincreasing and convex in the number of servers, $m$, indicating that marginal analysis is optimal for determining server allocation among various service facilities so as to minimize the total expected queueing time. The proofs in \cite{vasicek1977inequality,hajek1983proof,weber1980note} rely on convexity properties. 
Convexity of queueing performance metrics is a highly desired property in the parametric optimization of stochastic systems. A substantial body of research employs direct algebraic methods to establish convexity of the steady-state performance metrics as functions of the design parameters; see, for example, \cite{grassmann1983convexity,lee1983note,jagers1986continued,harel1987convexity,harel1990convexity2,harel1990convexity,weber1983note}. 
For a comprehensive discussion on the related notion of stochastic convexity, we refer the interested reader to \cite{shaked1988stochastic,shanthikumar1988strong,shaked2007stochastic}. However, as mentioned earlier, for the present study we require convexity of the objective function's derivative. Several closely related concepts emerge in the extremal analysis of queueing systems. In this context, the need for a convex derivative is replaced by closely related sufficient conditions that warrant the application of the theory of complete Tchebychev systems \cite{karlin1966tchebycheff}. For a detailed discussion of the relevance of these systems to queues, please refer to \cite{eckberg1977sharp,chen2020extremaltail,gupta2011markov}.


{\it Rational queueing with limited information.}
A mature literature exists \citep{hassin2016rational} that seeks to elucidate the rational decision-making processes of users who decide whether or not to access delay-sensitive services based on utility. \citet{naor1969regulation} initiated this line of work for the observable M/M/1 queue with a customer utility that depends linearly on the price, value of service, and the expected \new{waiting} time. \citet{edelson1975congestion} investigated similar questions for the unobservable M/M/1 \new{queue}. These works explore the environment with a deterministic arrival rate. \citet{liu2019naor} expanded upon Naor's model by assuming that the arrival rate is drawn from a known probability distribution accessible to the decision maker. This concept was further extended to the unobservable setting in \cite{chen2020knowledge}. Their research demonstrated that a socially optimal pricing strategy results in a lower expected arrival rate compared to a price set by a revenue-maximizing decision maker. \citet{hassin2021strategic} examined the unobservable model with a random arrival rate from a distinct angle, where strategic customers base their decisions on a rate-biased time-average property when estimating their expected waiting time. \citet{wang2022distributionally} extended the traditional observable model to a distributionally robust setting by taking into account an uncertain arrival rate governed by an unknown underlying probability distribution.

\section{Tight bounds for expected waiting time with limited market knowledge}\label{MMSsec2}
In this section, we derive bounds for the expected \new{waiting} time with an uncertain arrival rate. Section~\ref{subsec:themodelMMS} introduces \new{the basic problem of finding a tight bound for the expected waiting time.} Section~\ref{subsec:forgenericfunc} \new{states} the second-order bounds for general convex functions with a convex derivative. In Section~\ref{subsec:mainlambdaresult}, we apply these second-order bounds to the expected \new{waiting} time in the M/M/$s$ queue. Section~\ref{subsecUnimod} incorporates unimodality information to refine the \new{second-order bounds}. In Section~\ref{sec:furtherapplic}, we explore the extension wherein the service rate parameter is likewise a random parameter. Further, \new{as an example of the broader applicability of our approach}, we \new{obtain} second-order bounds for the expected waiting time in the M/G/1 queue with random arrival rate. 


\subsection{The model}\label{subsec:themodelMMS}
Consider an M/M/$s$ queue with Poisson arrivals with rate $\lambda$, unit mean exponential service requirements, and $s$ parallel servers, each operating with service rate $\mu$. Let $W(s,\lambda)$ denote the expected waiting time experienced by a customer in this queueing system (including the time in service), which depends on the number of servers $s$ and the arrival rate $\lambda$. For notational convenience, we omit at first the functional dependence on the service rate parameter $\mu$, but our findings remain valid for any given value of this parameter. Suppose that $\lambda/\mu<s$. Then the single-server M/M/1 queue has as expected waiting time $W(1,\lambda)=(\mu-\lambda)^{-1}$. Observe that
$$
\frac{\partial}{\partial \lambda} W(1,\lambda) = \frac{1}{(\mu-\lambda)^2} > 0,\quad  \frac{\partial^2}{\partial \lambda^2} W(1,\lambda) = \frac{2}{(\mu-\lambda)^3} > 0
$$
and
$$
\frac{\partial^3}{\partial \lambda^3} W(1,\lambda) = \frac{6}{(\mu-\lambda)^4} > 0, \quad \text{for } \lambda<\mu.
$$
Hence, $\lambda\mapsto W(1,\lambda)$ is an increasing, convex function of $\lambda$ with a convex derivative. We now present some general results for functions $\phi$ that have the same properties as $W(1, \cdot)$. We present the results in general form, because we will apply these results later to $W(s,\lambda)$ with $s\geq 2$ as well, to obtain the solutions for
 \begin{align}\label{eq:themainmaxproblem}
 \max_{\mathbb{P}\in \mathcal{P}_{(m,\sigma)}} \mathbb{E}_\mathbb{P}[W(s,\Lambda)], 
\end{align}
in which we focus on the case $\cP=\mathcal{P}_{(m,\sigma)}$ with the latter being shorthand notation for $\mathcal{P}(m,\sigma,\llambda,\ulambda)$, the set of all distributions with mean $m$, standard deviation $\sigma$ and the support contained in the interval $[\llambda,\ulambda]$. That is,
$$
\mathcal{P}(m,\sigma,\llambda,\ulambda):= \left\{\P\in\cP_0([\llambda,\ulambda]) : \E_\P[\Lambda]=m,\ \E_\P[\Lambda^2]=\sigma^2+m^2 \right\},
$$
where $\cP_0([\llambda,\ulambda])$ comprises all probability distributions with support contained in $[\llambda,\ulambda]$. We shall refer to such a set as an ambiguity set. Henceforward, the market size is viewed as a random variable with a distribution contained in $\mathcal{P}(m,\sigma,\llambda,\ulambda)$. We further assume from this point onward that the system under consideration is stable for each realization of the random arrival rate; that is, $\ulambda<s\mu$.

 
\subsection{Bounding the expectation of a function with convex derivative}\label{subsec:forgenericfunc}
\new{We first develop a general bound for $\int_{x} \phi(x) \mathrm{d}\P(x)$ with $\phi(\cdot)$ being a function with a convex derivative, just like the function $W(s,\cdot)$.
We therefore solve $\max_{\P\in\cP_{(m,\sigma)}}\E_\P[\phi(X)]$ where the constraints that define $\cP_{(m,\sigma)}$ correspond to known first and second moments and support contained in $[\lx,\ux]$. We thus need to solve}
\begin{equation}\label{eq:ThePrimal}
\begin{aligned}
 \max_{\P \in \cP_0([\lx,\ux])} & \int_{x} \phi(x) \mathrm{d}\P(x) \\
\textup{s.t.}\quad & \int_{x} x \mathrm{d}\P(x)=m, \\
& \int_{x} x^2 \mathrm{d}\P(x)=\sigma^2 + m^2,
\end{aligned}
\end{equation}
 a semi-infinite linear program with an (infinite-dimensional) decision variable, the probability distribution $\P$, which must reside within $\cP_{(m,\sigma)}$.
{Since there are three constraints (including  $\smallint_{x} \mathrm{d}\P(x)=1$), it is known from general optimization theory (see, e.g.,~\cite[Theorem~1]{rogosinski1958moments}, \cite[Lemma~3.1]{shapiro2001duality}, \cite{smith1995generalized}) that among the optimal solutions of  \eqref{eq:ThePrimal} there is a distribution with at most three points of support. 
We next prove a stronger result which says that among the optimal solutions to \eqref{eq:ThePrimal}, there is a distribution with two points.
The proof of this result uses the dual problem of 
\eqref{eq:themainmaxproblem}, defined as finding the vector of dual variables $(\pi_0, \pi_1, \pi_2)$ that solves
\begin{equation}\label{eq:TheDual}
\begin{aligned}
 \min_{\pi_0, \pi_1, \pi_2} & \pi_0+\pi_1 m+\pi_2 (\sigma^2+m^2) \\
\textup{s.t.}\quad & M(x):=\pi_0+\pi_1 x+\pi_2 x^2 \geqslant \phi(x), \quad \forall x \in [\lx,\ux], \\
\end{aligned}
\end{equation}
and uses that the objective function $\phi(\cdot)$ has a convex derivative. Denote the optimal dual solution as $(\pi^*_0, \pi^*_1, \pi^*_2)$.

\begin{lemma}\label{lemma:cs}
Suppose that $m\in(\lx,\ux)$ and $\sigma^2\in(0,(\ux-m)(m-\lx))$. The probability distribution that solves 
\eqref{eq:ThePrimal} is supported on the points at which the function $M^*(x)=\pi^*_0+\pi^*_1 x+\pi^*_2 x^2$ coincides with $\phi(x)$.
\end{lemma}
\begin{proof}
According to Proposition~3.4 in \cite{shapiro2001duality}, problems \eqref{eq:ThePrimal} and \eqref{eq:TheDual} have 
the same optimal value when
the moment vector $(m,\sigma)$ is an interior point of the set of feasible moment vectors, which is implied by the stated conditions $m\in(\lx,\ux)$ and $\sigma^2\in(0,(\ux-m)(m-\lx))$.  Furthermore, \cite[Proposition~3.4]{shapiro2001duality} asserts that, in this case, the optimal dual solution is attained if the common optimal value is finite. Since $\phi(x)$ is continuous on the compact support $[\lx,\ux]$, it readily follows that the optimal value is indeed bounded. Hence, the dual optimal solution is attained. Further, \cite[Corollary 3.1]{shapiro2001duality} says that if the support is compact, the moment functions and the objective function are continuous, and the primal problem is consistent (i.e., the ambiguity set $\cP_{(\mu,\sigma)}$ is nonempty), then there is no duality gap and the optimal primal solution is attained. Under the stated conditions on the moment vector $(\mu,\sigma)$, the primal problem is indeed guaranteed to admit a feasible distribution. Thus, it immediately follows from the continuity of $x,$ $x^2,$ and $\phi(x),$ that the primal optimal solution is also attained.
Therefore, the optimal solutions of both \eqref{eq:ThePrimal} and \eqref{eq:TheDual} are attained, and there is no duality gap. This, in turn, implies that complementary slackness holds \cite[Proposition~2.1]{shapiro2001duality} and completes the proof.
\end{proof}

\begin{proposition}[Extremal two-point distributions]\label{propbirge}
Let $x\mapsto \phi(x)$ be a continuously differentiable function with its derivative $\phi'(\cdot)$ being either convex or concave on the bounded interval $[\lx,\ux]$. Suppose that $m\in(\lx,\ux)$ and $\sigma^2\in(0,(\ux-m)(m-\lx))$. Then for a random variable $X$ with distribution $\P\in\mathcal{P}(m,\sigma,\lx,\ux)$, the tight upper bound of $\E_\P[\phi(X)]$ is attained by a discrete distribution with two support points.  
\end{proposition}
\begin{proof}

It is evident that the optimal distribution cannot be a one-point distribution, since $\sigma>0$, so that $M(\cdot)$ is bound to intersect with $\phi(\cdot)$ more than once by Lemma~\ref{lemma:cs}. We first consider the case with $\phi'(\cdot)$ being either strictly convex or strictly concave. It suffices to demonstrate that a candidate dual function $M^*$ coincides with $\phi$ at precisely two points, rather than three. 
Figure~\ref{figvarderiv} illustrates the functions $M$ and $\phi$, and their derivatives.

\begin{figure}[h]
\begin{subfigure}{.49\linewidth}
\centering
\includegraphics[width=.75\linewidth]{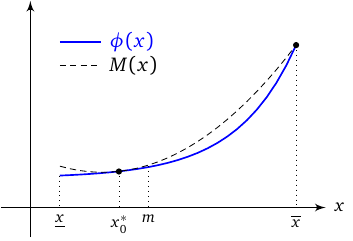}
\caption{}\label{figvarderiva}
\end{subfigure}
\begin{subfigure}{.49\linewidth}
\centering
\includegraphics[width=.75\linewidth]{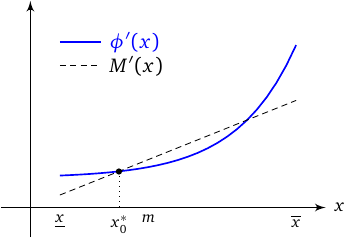}
\caption{}\label{figvarderivb}
\end{subfigure}
\caption{Objective function $\phi$, dual function $M$ and their derivatives}
\label{figvarderiv}
\end{figure}

We proceed by contradiction. Suppose that there exists a feasible $M(\cdot)$ such that it coincides with $\phi(\cdot)$ at three points.
Let the first intersection point of the derivatives in Figure~\ref{figvarderivb} correspond to a tangent point for the original functions. Due to dual feasibility, $M'(\cdot)$ has to be greater than $\phi'(\cdot)$ directly after the tangent point, and $M(\cdot)$ can only be tangent to $\phi(\cdot)$ once since $M'(\cdot)$ is linear and $\phi'(\cdot)$ is strictly convex. Hence, again by dual feasibility, there cannot exist a second intersection point in the interior of $[\lx,\ux]$. Therefore, a candidate dual solution that coincides with $\phi(\cdot)$ thrice should intersect once in the interval $(\lx,\ux)$ and at both boundary points. Let $M_{x^*_0}(\cdot)$ be such that it is tangent to $\phi(\cdot)$ at some $x_0^*$ and intersects at $\ux$. Assume now that $M_{x^*_0}(\cdot)$ also touches $\phi(\cdot)$ at $x=\lx$. This implies that $M_{x^*_0}'(\lx)\geq\phi'(\lx)$ for $M_{x^*_0}$ to be dual feasible. However, as shown in Figure~\ref{figvarderiv}, this cannot hold, for $\phi'$ is strictly convex. Hence, we arrive at a contradiction. Likewise, for strictly concave $\phi'(\cdot)$, we construct $M_{x^*_0}(\cdot)$ such that it touches at $\lx$, and it is easily shown that imposing this function to intersect also at $\ux$ leads to a similar contradiction.

To complete the proof, we must still consider the case where $\phi'(\cdot)$ is either convex or concave, but not strictly so. This implies that $\phi'(\cdot)$ is linear over certain intervals of its domain. Suppose we are provided with a dual feasible function $M_{x_0}$. 
If $M'_{x_0}$ does not align with $\phi'$ on any of these intervals, the preceding argument remains valid. However, if $M'_{x_0}$ coincides with $\phi'$ on such an interval, $M_{x_0}$ and $\phi$ must match over the entire interval $[x_0-\delta, \ux]$, for some $\delta\geq0$, by dual feasibility. This condition results in a continuum of potential choices for the support of the extremal distribution, as dictated by complementary slackness. Nevertheless, this subsumes the situation of a two-point support. Therefore, the assertion presented in the claim holds true. An analogous argument applies for a concave $\phi'(\cdot)$, completing the proof of the claim.  
\end{proof}


Proposition~\ref{propbirge} partly overlaps with  
Theorem 5.1 in \cite{birge1991bounding}. Proposition~\ref{propbirge} stretches the conditions in
\cite[Theorem 5.1]{birge1991bounding}, as \cite[Theorem 5.1]{birge1991bounding} requires $\phi(\cdot)$ to be convex. 
On the other hand, \cite[Theorem 5.1]{birge1991bounding} allows the derivative $\phi'(x)$ to be convex on the interval $[\lx,c]$ and concave on $[c,\ux]$ for some $c\in(\lx,\ux)$. For the purpose of this paper, \cite[Theorem 5.1]{birge1991bounding} would have sufficed, but we present Proposition~\ref{propbirge} as a slight generalization and to provide a self-contained exposition. 


Proposition~\ref{propbirge} defines a class of objective functions for which the mean-variance information set leads to two-point worst-case distributions. The defining feature of this class of functions is the convex derivative. Based on general optimization theory, we knew already that three points suffice for describing the worst-case distribution, but with Proposition~\ref{propbirge} this is reduced to two points. }
It is thus enough to optimize over two-point distributions in the search for the extremal distribution that attains the tight bounds for $\E_\P[\phi(X)]$. 
{\citet[Chapter~4]{kreuin1977markov} demonstrate that, when $\phi$ is continuously differentiable with a strictly convex derivative on $[\lx,\ux]$, a unique two-point distribution solves the moment problem \eqref{eq:ThePrimal}. The following lemma establishes a similar result.}

\begin{lemma}\label{lemma:ec1}
Consider a function $x\mapsto \phi(x)$ that is continuously differentiable on the bounded interval $[\lx,\ux]$, and let $X$ be a random variable with distribution $\P\in\mathcal{P}(m,\sigma,\lx,\ux)$. \\
{\rm (i)} If the derivative $\phi'(x)$ is convex on $[\lx,\ux]$, the tight upper bound for $\E_\P[\phi(X)]$ is attained by a two-point distribution with values
$
\{m-\frac{\sigma^2}{\ux-m},\ux\}
$ and corresponding probabilities 
$
\{\frac{(\ux-m)^2}{(\ux-m)^2+\sigma^2},\frac{\sigma^2}{(\ux-m)^2+\sigma^2}\}
$.\\
{\rm (ii)} If the derivative $\phi'(x)$ is concave on $[\lx,\ux]$, the tight upper bound for $\E_\P[\phi(X)]$ is attained by a two-point distribution with values
$
\{\lx,  m + {\frac{\sigma^2}{m-\lx}}\}
$
with corresponding probabilities 
$
\{\frac{\sigma^2}{(m-\lx)^2+\sigma^2},\frac{(m-\lx)^2}{(m-\lx)^2+\sigma^2}\}
$.
\end{lemma}

\begin{proof}
A two-point distribution with mean $m$ and variance $\sigma^2$ has support
\begin{equation}\label{eq:supp}
x_1 = m - \sqrt{\frac{\alpha}{1-\alpha}}\sigma, \quad x_2 = m + \sqrt{\frac{1-\alpha}{\alpha}}\sigma
\end{equation}
with probabilities $\alpha,\ 1-\alpha$, respectively.
We thus need to solve
$
\max_{\alpha} \Phi(\alpha)
$
with
$$
\Phi(\alpha):=\alpha\phi\left(m + \sqrt{\frac{1-\alpha}{\alpha}}\sigma\right) + (1-\alpha) \phi\left(m - \sqrt{\frac{\alpha}{1-\alpha}}\sigma\right)
$$
and
$$
\alpha\in\Big[\frac{\sigma^2}{(\ux-m)^2+\sigma^2}, \frac{(m-\lx)^2}{(m-\lx)^2+\sigma^2}\Big]
$$
due to the support being contained in  $[\lx,\ux]$. 
 To show that $\alpha^*=\frac{\sigma^2}{(\ux-m)^2+\sigma^2}$ is a maximizer, we will prove that $\Phi(\alpha)$ is nonincreasing in $\alpha$. For differentiable $\phi$, we have that
\begin{equation}\label{eq:optcond}
\frac{{\rm d}}{{\rm d}\alpha}\Phi(\alpha) = \phi(x_2) - \phi(x_1) - (x_2-x_1)\frac{\phi'(x_1)+\phi'(x_2)}{2}.
\end{equation}
In order for $\Phi(\alpha)$ to be nonincreasing, we need
\begin{align*}
    \phi(x_2) - \phi(x_1) - (x_2-x_1)\frac{\phi'(x_2)+\phi'(x_1)}{2} \leq 0,
\end{align*}
or
$$
\frac{\phi'(x_2)+\phi'(x_1)}{2} \geq \frac{\phi(x_2) - \phi(x_1)}{x_2-x_1}.
$$
This inequality holds by convexity of $\phi'$.
Thus, as $\Phi(\alpha)$ is nonincreasing in $\alpha$, the maximum occurs at the lower bound of the feasible set (i.e., $\alpha^*=\frac{\sigma^2}{(\ux-m)^2+\sigma^2}$). Substituting the optimal value $\alpha^*$ into our parameterized expression for the two-point distribution then yields assertion (i).

For assertion (ii), note that for concave $\phi'$, $\Phi(\alpha)$ is nondecreasing in $\alpha$ and therefore maximized by $\alpha^*=\frac{(m-\lx)^2}{(m-\lx)^2+\sigma^2}$.
\end{proof}

Notice that assertion (ii) yields the extremal distribution that solves $\min_{\P\in\cP_{(m,\sigma)}}\E_\P[\phi(X)]$, since determining the lower bound is tantamount to maximizing $\E_\P[-\phi(X)]$.
Returning to the expected waiting time in the M/M/1 queue, we see that the tight bounds for
$\mathbb{E}_\mathbb{P}[ W(1,\Lambda)]$ are attained by the two-point distributions stated in Lemma~\ref{lemma:ec1}. We shall prove this result for the multi-server system in the next subsection.

\new{Let us mention two other streams of literature that also deal with tight bounds for stochastic models, and in particular three- and two-point worst-case distributions. 
The first stream stems from statistics. 
\cite{guljavs1998jensen} presents general results for the maximization of integrals over distributions of which only the 
first two moments are given, and where for convex integrands the maximizing distribution turns out to be a three-point distribution. Among the many statistical applications of the results in \cite{guljavs1998jensen}  are bounding the expected sample maximum and the expected sample range, where for the latter, \cite{hartly1954universal} claims that the extremal distribution can be reduced to a two-point distribution. The other line of research in which three- and two-point extremal distributions feature prominently deals with establishing tight bounds for the waiting time moments in the general GI/G/1 queue, or for more specific queues such as M/G/1 and G/M/1 queues; see \cite{whitt1984approximations,eckberg1977sharp,rolski1972some,pearce1995integral,van2022mad} among others.
In particular, a long-standing conjecture for the expected waiting time in GI/G/1 queue says that the tight bounds follow from the two-point distributions stated in Lemma~\ref{lemma:ec1}. While recently some progress has been made, this conjecture is still open \cite{chen2022correction,chen2021extremal}.}

\subsection{Bounding the expected waiting time as function of the arrival rate}\label{subsec:mainlambdaresult}

To apply the results in the previous section to the multi-server M/M/$s$ queue with random arrival rate, we must show that the expected waiting time, as a function of the arrival rate, has a convex derivative. Therefore, we need to study the third derivative of $W(s,\lambda)$ with respect to $\lambda$. The expected waiting time is a classic formula in the queueing literature, for which a host of properties have been discovered, including convexity with respect to $\lambda$; see, e.g., \cite{grassmann1983convexity,lee1983note}. 
But to the best of our knowledge, the property of a convex derivative with respect to $\lambda$ has not been reported. However, this property was reported in \cite{randhawa2016optimality} for the expected queue size defined as
\begin{equation}\label{eq:L0thder}
L(s,\lambda) := s\rho + \frac{\rho}{1-\rho} C(s,\lambda)
\end{equation}
where $C(s,\lambda)$ is the Erlang-C formula, the probability that an arriving customer experiences delay in the M/M/$s$ queue, and $\rho=\lambda/(s\mu)$. 
\citet{randhawa2016optimality}  proved that $L(s,\lambda)$ as function of the arrival rate $\lambda$, indeed has a convex derivative.
By Little's law, $W(s,\lambda)=L(s,\lambda)/\lambda$, and one could try to see whether the convex derivative of
$L(s,\lambda)$  implies the same property for 
$W(s,\lambda)$. Unfortunately, we do not see a direct proof for this argument, in which we can use convexity of the derivative of the expected queue length $L$ to demonstrate convexity of the derivative of the expected waiting time $W$.
We therefore choose to provide a standalone proof for the fact that $W(s,\lambda)$ has a convex derivative, just like $L(s,\lambda)$. 
The proof method we apply is largely inspired by that in 
\cite{randhawa2016optimality}. In the following, we provide a proof sketch. The details are relegated to the Appendix. First, we determine an expression for the third derivative of the expected waiting time as a function of the server utilization $\rho$. Then, we write this expression as a function of the Erlang-C formula $C$ and the number of servers $s$. It then suffices to show that this expression is nonnegative for all values of $s$ and $\forall\rho\in(0,1)$, which implies convexity of $W(s,\lambda)$ as a function of $\lambda$ with the number of servers $s$ fixed.



\begin{lemma}[Derivative w.r.t. $\lambda$ is convex]
The expected waiting time $W(s,\lambda)$ as function of $\lambda\in(0,\mu)$
has a convex derivative.  
\end{lemma}

Convexity of $W(s,\cdot)$ as a function of $\lambda$ implies that the extremal distribution of $\Lambda$ is given by the extremal distribution in Lemma~\ref{lemma:ec1}(i), thus yielding our main result.

\begin{theorem}[Second-order bound for random arrival rate]\label{thmlambda2}
Consider an {M/M/$s$} queue with random arrival rate $\Lambda$ that follows a distribution $\mathbb{P}$ belonging to the ambiguity set  $\mathcal{P}(m,\sigma,\llambda,\ulambda)$. The tight upper bound for the expected waiting time 
$\mathbb{E}_{\mathbb{P}}[{W}(s,\Lambda)]$ is attained by the two-point distribution with support 
$$
\lambda_1=m - \frac{\sigma^2}{\ulambda - m},\quad \lambda_2=\ulambda,
$$
and corresponding probabilities 
$$
p_1=\frac{(\ulambda-m)^2}{(\ulambda-m)^2+\sigma^2} ,\quad p_2=\frac{\sigma^2}{(\ulambda-m)^2+\sigma^2}.
$$
\end{theorem}

We demonstrate the effectiveness of our second-order bounds through a numerical experiment. The results are presented in Table~\ref{tab:EW}. We assume that the “true” distribution of the random arrival rate follows a beta distribution with support $[0,\ulambda]$ and standard deviation fixed to $\sigma=0.1$. We consider the scenario with $s=2$ servers and vary the mean $m$ to obtain results for multiple utilization levels. The average utilization level is denoted by $\rho=m/(s\mu)$. As shown in Table~\ref{tab:EW}, the second-order bounds proposed in Theorem~\ref{thmlambda2} perform well for low- and medium-utilization regimes. However, for high utilization, the bounds deviate significantly from the true value when assuming a beta distribution for $\Lambda$. Furthermore, the bounds widen as the upper bound of the support $\ulambda$ increases. The second column corresponds to the second-order lower bound, which can be obtained using the extremal distribution in Lemma~\ref{lemma:ec1}(ii). Notice that this bound depends on $\llambda$, the lower bound of the support, rather than $\ulambda$.

\begin{table}[h!]
\renewcommand{\arraystretch}{1}
\centering
\caption{Numerical bounds for the expected waiting time in the M/M/2 queue with random arrival rate following a beta distribution with support $[0,\ulambda]$ and $\sigma=0.1$}\label{tab:EW}
\begin{tabular}{@{}cccc|cc|cc@{}}
\toprule
       &         & \multicolumn{2}{c}{$\ulambda=1.925$} & \multicolumn{2}{c}{$\ulambda=1.95$} & \multicolumn{2}{c}{$\ulambda=1.99$} \\ \cmidrule(l){3-8} \cmidrule(l){3-8} 
$\rho$ & LB      & $\E[W(2,\Lambda)]$  & \textrm{UB}        & $\E[W(2,\Lambda)]$  & \textrm{UB}      & $\E[W(2,\Lambda)]$  & \textrm{UB}       \\ \midrule
0.2    & 1.0445 & 1.0449             & 1.0940  & 1.0449            & 1.1199  & 1.0449            & 1.4312   \\
0.5    & 1.3389 & 1.3440             & 1.4655  & 1.3440            & 1.5315  & 1.3440            & 2.3236   \\
0.7    & 1.9753 & 2.0094             & 2.3200  & 2.0096            & 2.5002  & 2.0098            & 4.6625   \\
0.8    & 2.8099 & 2.9442             & 3.5518  & 2.9465            & 3.9446  & 2.9497            & 8.6515   \\
0.9    & 5.3921 & 6.4285             & 7.6443  & 6.6019            & 9.0133  & 6.9387            & 25.0550  \\
\bottomrule
\end{tabular}
\end{table}

\subsection{Sharper bounds for unimodal market size}\label{subsecUnimod}
Often more is known about the market-size distribution than just the mean and variance. For example, we may have some structural information regarding the distribution of the market size. In this section, we show how to use this structural information to improve the tight bounds, corresponding to the situation in which the service provider has the additional information that the market-size distribution is unimodal, which can be understood as there being one primary market segment that constitutes the bulk of the arrivals requesting service. 

A random variable is said to be unimodal with mode $\mode$ if its distribution can be characterized as the mixture of a Dirac distribution $\delta_{\mode}$ and a distribution function that is a concave function on $[\lx,m]$ and a convex function on $(m,\ux]$. Some examples of unimodal distributions are the normal, exponential, and beta probability distributions. In order to incorporate such unimodality, we apply Khinchine's characterization theorem, which states that a random variable $Y$ has a unimodal distribution with mode zero if, and only if, there exists a random variable $Y$ that can be expressed as $Y=UV$, where $U$ is a uniformly distributed random variable on $[0,1]$ independent of some random variable $V$ (see, e.g., \citep[p.~158]{feller1971}). We adopt the following approach from \cite{kemperman1971moment} (see also \cite{brockett1985insurance}). First, we transform the problem from the unimodal random variable $X$ to the auxiliary variable $V$. Then, we establish tight bounds for $V$ and use them to obtain bounds for $X$. To be specific, if $X$ is unimodal with mode $\mode$, then $Y=X-\mode$ is also unimodal but with mode 0. Hence, according to Khinchine's theorem, we have $Y=UV$, where $U$ is uniformly distributed over $[0,1]$. According to Kemperman's approach, we can transform the moment problem in terms of the random variable $X$ into one in terms of the auxiliary random variable $V$ using that for any function $\phi$, $\E[\phi(Y)]=\E[\phi^*(V)]$, where
\begin{equation}\label{eq:uniformExpectation}
\phi^*(x)=\frac{1}{x} \int_0^x \phi(t) d t=\E[\phi(U V) \mid V=x].
\end{equation}
By taking $\phi(x)=x^k$, we can relate the moments of $Y$ to those of $V$ through
$$
\phi^*(x)=\frac{1}{(k+1)} x^k,
$$
so that $\E[V^k]=(k+1) \E[Y^k]$. The mean and variance of $V$, in terms of the mean and variance of $X$, are given by
$$
\begin{aligned}
m_V & =\E[V]=2 \E[Y]=2\left(m_X-\mode\right). \\
\sigma_V^2 & =\E[V^2]-(\E[V])^2=3 \E[Y^2]-4(\E[Y])^2=3 \sigma_Y^2-(\E[Y])^2=3 \sigma_X^2-\left(m_X-\mode\right)^2.
\end{aligned}
$$
We can use the same integral relationship to find tight bounds for $\E[\phi(X)]$ when $X$ is unimodal with mode $\mode$. Specifically, this problem is equivalent to finding bounds for $\E[\phi^*(V)]$ subject to the moment constraints on $V$, where
$$
\phi^*(x)=\frac{1}{x} \int_0^x \phi(u+\mode) \d u.
$$
We already have the solution for this transformed problem, since the (conditional) expectation operator \eqref{eq:uniformExpectation} preserves the curvature properties. By substituting the transformed moments into the result stated in Lemma~\ref{lemma:ec1}, we obtain the following result.

\begin{theorem}[M/M/$\boldsymbol{s}$ queue with unimodality]\label{thmlambdaunimod}
Consider an {M/M/$s$} queue with random arrival rate $\Lambda$. Suppose that $\Lambda$ is unimodal with mode $\mode$, mean $m$, variance $\sigma^2$, and has support contained in $[\llambda,\ulambda]$. Let
$$
W^*(s,x)=\frac{1}{x} \int_m^{x+m} W(s,t)\d t.
$$
Then, the tight upper bound for the expected waiting time 
$\mathbb{E}_{\mathbb{P}}[{W}(s,\Lambda)]$ is given by
$$
(1-p)W^*\left(s,\lambda_1\right)+pW^*(s,\lambda_2), 
$$
where
$$
\begin{aligned}
\lambda_1=2 m-2 \mode-\frac{3 \sigma^2-\left(m-\mode\right)^2}{\ulambda+\mode-2 m},\quad \lambda_2=\ulambda - \mode,\quad p=\frac{3 \sigma^2-\left(m-\mode\right)^2}{3 \sigma^2-\left(m-\mode\right)^2+\left(\ulambda+\mode-2 m\right)^2}.
\end{aligned}
$$  
\end{theorem}

Table~\ref{tab:EWuni} demonstrates the performance of the tight mean-variance bounds that incorporate unimodality information. Evidently, incorporating this structural information significantly sharpens the bounds compared to the standard mean-variance bounds without this information. 

\begin{table}[h!]
\renewcommand{\arraystretch}{1}
\centering
\caption{Numerical bounds for the expected waiting time in the M/M/2 queue with unimodal beta arrival rate distribution}\label{tab:EWuni}
\begin{tabular}{@{}cccc|ccc@{}}
\toprule
       &    \multicolumn{3}{c}{$\ulambda=1.95$} & \multicolumn{3}{c}{$\ulambda=1.99$} \\ \cmidrule(l){2-7} \cmidrule(l){2-7} 
$\rho$      & $\E[W(2,\Lambda)]$  & \textrm{Var-UB}        & \textrm{Unimod-UB}  & $\E[W(2,\Lambda)]$      & \textrm{Var-UB}  & \textrm{Unimod-UB}       \\ \midrule
0.2    & 1.0449             & 1.1199  & 1.0583            & 1.0449  & 1.4312             & 1.0686  \\
0.5    & 1.3440             & 1.5315  & 1.3879            & 1.3440  & 2.3236             & 1.4277  \\
0.7    & 2.0096             & 2.5002  & 2.1502            & 2.0098  & 4.6625             & 2.3124  \\
0.8    & 2.9465             & 3.9446  & 3.2477            & 2.9497  & 8.6515             & 3.7098  \\
0.9    & 6.6019             & 9.0133  & 7.1468            & 6.9387  & 25.0550            & 9.3105  \\ 
\bottomrule
\end{tabular}
\end{table}

\subsection{Broader robust applications}\label{sec:furtherapplic}
In practice, it might turn out that other parameters besides the arrival rate are uncertain, with only partial information available. It seems reasonable to assume that decision makers may not always have precise knowledge of service rates, as these rates might depend on specific server characteristics. Hence, it is of interest to establish also second-order bounds for the service rate at which the different servers operate. In this regard, we consider the scenario in which the service rate $\mu$ is replaced by a random variable $M$. 
Convexity of the mean queue size of the M/M/$s$ queue with respect to the system utilization $\rho$ is a classic result \citep{grassmann1983convexity,lee1983note}. The following result pertains to the multi-server setting, where we again will rely on the results in
\cite{randhawa2016optimality}.

\begin{lemma}[Derivative w.r.t. $\mu$ is concave]\label{lemma:muconcave}
The expected waiting time $W(s,\lambda, \mu)$ as function of $\mu$
has a concave derivative.  
\end{lemma}

\begin{proof}


Denote $\phi_\mu(x):=W(s,\lambda,x)$. To prove that  $\phi_\mu'(x)$ is concave, we will use the result in \cite{randhawa2016optimality} which proves convexity of $L'$ as function of $\rho$.
Interpreting $L$ indeed as function of $\rho$, define $h(x):=L$ and $g(\mu):=\lambda/(s\mu)$, where $\mu$ is the variable and the parameters $\lambda$ and $s$ are fixed to arbitrary constants. To see that the composite function $(h\circ g)(\mu) = L(\lambda/(s\mu))$ is decreasingly convex, notice that
$$
\frac{\partial^3}{\partial \mu^3} (h\circ g)(\mu)= \underbrace{g'''(\mu) h'(g(\mu))}_{g'''<0, h'\geq0} + \underbrace{g'(\mu)^3 h'''(g(\mu)) }_{g'<0, h'''\geq0} + \underbrace{3g'(\mu)g''(\mu)h''(g(\mu))}_{g'<0,g''>0,h''\geq0} \leq 0.
$$
Little's law indicates that $W(s,\lambda,\mu) = L(s,\lambda,\mu) / \lambda$, with $\lambda>0$ a fixed constant, which completes the proof.
\end{proof}

Hence, $W(s,\lambda,\mu)$ is a decreasing, convex function of $\mu$ with a concave derivative. By Lemma~\ref{lemma:ec1}(ii), the extremal distribution is given by a two-point distribution with support $\{\mu_1,\mu_2\}=\{\lmu,m_{\mu}+\frac{\sigma^2}{m_{\mu}-\lmu}\}$ and respective probabilities 
$$
p_1=\frac{\sigma^2}{(m_{\mu}-\lmu)^2+\sigma^2}, \quad p_2=\frac{(m_{\mu}-\lmu)^2}{(m_{\mu}-\lmu)^2+\sigma^2}.
$$
That is, the worst case is given by servers which either operate at their lowest service capacity or just above the mean service rate.

Moreover, it is possible to work with multiple uncertain parameters simultaneously. We next consider
$$
 \max_{\mathbb{P}_\Lambda\in \mathcal{P}_{(m_\Lambda,\sigma_\Lambda)}} \max_{\mathbb{P}_M\in \mathcal{P}_{(m_M,\sigma_M)}} \mathbb{E}_{\P_M \otimes\P_\Lambda}[ W(s,\Lambda,M)],
$$
which turns out to be solvable in closed form:

\begin{theorem}[M/M/$\boldsymbol{s}$ queue with partially known service rate and market size]\label{thm1lambdamu}
Let the service rate $M$ follow a marginal distribution $\mathbb{P}_M\in\mathcal{P}_M:=\mathcal{P}_{(m_M,\sigma_M)}$ and the market size $\Lambda$ follow a marginal distribution $\mathbb{P}_\Lambda\in\mathcal{P}_\Lambda:=\mathcal{P}_{(m_\Lambda,\sigma_\Lambda)}$. Suppose that $\ulambda<s\lmu$. If $M$ and $\Lambda$ are independent, the sharpest possible upper bound for $\mathbb{E}_{\mathbb{P}}[W(s,\Lambda,M)]$ is attained by the product measure $\P^*_M\otimes\P^*_\Lambda$, with $\P^*_\Lambda$ and $\P^*_M$ as defined in {\rm Lemma~\ref{lemma:ec1}(i)} and~{\rm(ii)}, respectively.
\end{theorem}

\begin{proof}
The tower rule yields
$$
\max_{\mathbb{P}_M\in \mathcal{P}_M}\mathbb{E}_{\P}[ W(s,\Lambda,M)] = \max_{\mathbb{P}_M\in \mathcal{P}_M}\E_{M}[\mathbb{E}_{\Lambda}[ W(s,\Lambda,M)|M]], 
$$
where the expectation is well defined due to $s\lmu>\ulambda$. Hence, fixing the distribution of $\Lambda$ gives the moment problem 
$$
\max_{\mathbb{P}_M\in \mathcal{P}_M} \int_{[\lmu,\umu]} \E_{\Lambda}[W(s,\Lambda,\mu)] \Dist{\mu}.
$$
For $\lambda$ a fixed constant, $W(s,\Lambda,\mu)$ has the required convexity properties, so $\P^*_M$ is the solution to the moment problem. For $\P^*_M$ to work also for $\phi_\mu(x):=\E_{\Lambda}[W(s,\Lambda,x)]$, we need to check the signs of the first three derivatives. We interchange differentiation and expectation operations (which is allowed as $\Lambda$ is bounded), so that
$$
\phi^{(k)}_\mu(x) = \frac{\d^{(k)}}{\d x^{(k)}}\E_{\Lambda}[W(s,\Lambda,x)] = \E_{\Lambda}\left[\frac{\d^{(k)}}{\d x^{(k)}} W(s,\Lambda,x)\right].
$$ 
As the expectation operator is nonnegative,
$$
\frac{\d^{(k)}}{\d x^{(k)}} W(s,\lambda,x) \geq 0,\ \forall\lambda \implies \E_{\Lambda}\left[\frac{\d^{(k)}}{\d x^{(k)}} W(s,\Lambda,x)\right] \geq 0,
$$
and likewise,
$$
\frac{\d^{(k)}}{\d x^{(k)}} W(s,\Lambda,x) \leq 0,\ \forall\lambda \implies \E_{\Lambda}\left[\frac{\d^{(k)}}{\d x^{(k)}} W(s,\Lambda,x)\right] \leq 0,
$$
which proves that $\P^*_M$ maximizes $\phi_\mu(x)$ for any distribution of $\Lambda$. Substituting the two-point distribution gives the moment problem
$$
\max_{\mathbb{P}_\Lambda\in \mathcal{P}_\Lambda} \E_{\P_\Lambda}[p_1 W(s,\Lambda,\mu_1) + p_2 W(s,\Lambda,\mu_2)].
$$
Based on the same line of reasoning, this expression is maximized by $\P^*_\Lambda$.
\end{proof}

\new{Theorem~\ref{thm1lambdamu} leverages that the extremal distributions for univariate settings are independent of the precise form of $W(s,\lambda,\mu)$, and can thus be applied recursively when considering multiple uncertain parameters.}

The applications of the second-order bound extend beyond the M/M/$s$ queue and potentially include a vast range of uses. However, to apply the bounds, we need the \new{“third-order property” of a convex derivative}. 
As a concrete example, \new{consider} the Pollaczek-Khinchine formula for the expected waiting time in the M/G/1 queue given by
$$
\widetilde{W}(\lambda,\mu) = \frac{\rho+\lambda \mu \tilde{\sigma}^2}{2(\mu-\lambda)}+\mu^{-1},
$$
where $\tilde{\sigma}^2$ denotes the variance of the service times. Now observe that
$$
\frac{\partial}{\partial \lambda} \widetilde{W}(\lambda,\mu) = \frac{1+\mu^2\tilde{\sigma}^2}{2(\lambda-\mu)^2} \geq 0,\quad  \frac{\partial^2}{\partial \lambda^2} \widetilde{W}(\lambda) = \frac{1+\mu^2\tilde{\sigma}^2}{(\mu-\lambda)^3}\geq 0,
$$
and moreover,
$$
\frac{\partial^3}{\partial \lambda^3} \widetilde{W}(\lambda,\mu) = \frac{3(1+\mu^2\tilde{\sigma}^2)}{(\lambda-\mu)^4} \geq 0, \quad \text{for } \lambda<\mu.
$$
Hence, $\new{\widetilde{W}}(\cdot,\mu)$ is increasingly convex in $\lambda$, and $\E[\widetilde{W}(\Lambda,\mu)]$ can be bounded using Lemma~\ref{lemma:ec1}, yielding the following result.

\begin{proposition}[M/G/1 queue with variance information]\label{thmlambdaMG1var}
Let the market size $\Lambda$ have a distribution $\mathbb{P}$ that resides in the ambiguity set  $\mathcal{P}(m,\sigma,\llambda,\ulambda)$, and let $\widetilde{W}(\lambda,\mu)$ denote the expected waiting time in the M/G/1 queue. Then, the tight upper bound for $\mathbb{E}_{\mathbb{P}}[\widetilde{W}(\Lambda,\mu)]$ is attained by the two-point distribution stated in {\rm Theorem~\ref{thmlambda2}}.
\end{proposition}

Hence, despite some \new{potentially} tedious algebra, the second-order bounds that incorporate variance information have the potential to be widely applicable to other stochastic systems. Notwithstanding, as a cornerstone of queueing theory, the M/M/$s$ queue already boasts numerous pertinent applications. These include, but are not limited to, staffing optimization problems and rational queueing models, both of which find extensive usage in service operations management. We shall delve into the topic of rational queueing in the next section.

\section{Rational queueing model}\label{MMSsec3}
We now apply the results derived in the previous section to rational queueing models. Section~\ref{subsecrationalknown} introduces the rational queueing model with a known market size. In Section~\ref{subsecrationalunknown}, we extend this model to the setting in which customers are ambiguity-averse about the total market size. We provide a numerical example in  Section~\ref{sec:numerics}.

\subsection{Model with known market size}\label{subsecrationalknown}
Consider a firm that sells delay-prone services to a market of rational delay-sensitive individuals. Individuals value service, but dislike waiting, and will only join when the net service value exceeds the wait costs. They cannot observe real-time queues and instead estimate their expected waiting time costs based on beliefs or information about the total arrival rate of all potential customers. The arrival rate or market size is measured in terms of the scale parameter $\lambda$, so that the expected time between arrivals of consecutive individuals is $1/\lambda$. All individual joining decisions together give an equilibrium arrival rate or market share $\lambda_e$, which could be viewed as the product of $\lambda$ times the probability that an individual decides to join. 
We then consider the firm as a price-setting monopolist seeking to maximize revenue. With $p$ the price for service, the firm can influence the net service value of individuals, and hence the joining probability. 
A low price yields small margins but high joining probability, while a high price increases the profit per customer but suppresses the market share. The firm should thus strike the optimal balance between profit per customer and market share. Such challenges have been thoroughly addressed in the rational queueing literature \citep{hassin2016rational}.

Assume that the firm operates according to an M/M/$s$ queue, but that the arrival process consists of individuals that decide to join-or-not based on a rational-choice model. An individual's utility in acquiring service is defined as 
$$
U=r-p-c W(s,\lambda_e)
$$
with $r$ the value of service, $p$ the price for service, $c$ the wait costs per time unit, and $\lambda_e$ the effective arrival rate of joining individuals. Here we assume that individuals cannot observe real-time queue lengths, but instead know how the expected waiting time varies with overall demand. 
Hence, the rational choices are based on a linear relation between net service value and expected waiting time costs, and an individual will only join when $U\geq 0$. 
This is a common set-up in the rational queueing literature with self-interested individuals that tend to overcrowd systems, because they ignore externalities---inconvenient side effects---that their decisions have on others. When an individual decides to join, this will increase congestion and wait for all customers. Standard game theory then predicts that all individual decisions together will lead to an equilibrium, expressed in the equilibrium arrival rate, the probability $q(p)$ that an arbitrary individual joins the systems multiplied by the market size $\lambda$. Since not joining the queue gives zero utility, the equilibrium joining probability $q(p)$ then solves $U=r-p- c W(s,q(p)\lambda)=0$. The arrival process of those who decide to join is then a Poisson process with rate $q(p)\lambda$. 
To optimize revenue, the firm needs to set the price that strikes the best balance between margin $p$ and market share $\lambda_e(p) =q(p)\lambda$. 
Notice that 
the joining probability $q(p)$ is an implicit function that decreases with $p$ and can only be determined as the solution to $U=0$ when individuals know the arrival rate parameter $\lambda$. The firm then effectively
solves the monopoly pricing problem,
$
\max_p p\cdot q(p)\cdot \lambda,
$
to attain maximal revenue. The optimal solution to this problem for a given market size is already well established \citep{hassin2016rational}. Instead, in the next subsection, we choose to view the market size as a random variable of which only limited information is available. 

\subsection{Model with ambiguous market size}\label{subsecrationalunknown}
We interpret the market size as an unknown parameter about which the individuals that constitute the market form beliefs. The parameter $\lambda$ is then replaced by a random variable $\Lambda$ with a distribution representing what individuals believe or know about the market size. We primarily focus on the situation where individuals only know the support, mean, and variance of $\Lambda$. From the individual's perspective, the variance of $\Lambda$ expresses uncertainty about the market size, and hence uncertainty about the expected waiting time. For this setting with partially-informed customers, or an ambiguously specified arrival process, we consider the revenue-maximizing firm as a maximin decision maker, first determining the worst-case market (a minimization problem), and then maximizing the revenue by selecting the best price for this worst-case market. 
The firm thus attempts to solve the maximin problem
\begin{equation}\label{mpp}
\max_p\min_{\mathbb{P}\in \mathcal{P}_{(m,\sigma)}} \mathbb{E}_\mathbb{P}[p\cdot\Lambda_e(p)]
\end{equation}
with $\mathbb{P}$ the distribution of the market size $\Lambda$
and 
$\mathcal{P}_{(m,\sigma)}$ the ambiguity set that contains all distributions that satisfy the partial information, given by the mean, variance, and support.  Minimax and maximin optimization problems such as \eqref{mpp} arise naturally in decision making under uncertainty. Our strategy for solving such problems will be to solve first the minimization problem $\min_{\mathbb{P}\in \mathcal{P}_{(m,\sigma)}} \mathbb{E}_\mathbb{P}[\Lambda_e(p)]$, and then the maximization problem for this worst-case market.
The technical challenge is to solve the minimization problem, which requires determining tight bounds for the expected waiting time for all distributions of $\Lambda$ that comply with the partial information. However, this part was already resolved in the previous section, where we derived an upper bound on the expected waiting time. A tight upper bound on the expected waiting time presents the worst-possible scenario in the rational choice model and yields the corresponding tight lower bound on the market share. 
Observe from the expression for the utility $U\geq0$ that solving $\min_{\mathbb{P}\in \mathcal{P}_{(m,\sigma)}} \mathbb{E}_\mathbb{P}[\Lambda_e(p)]$
is tantamount to solving
$\max_{\mathbb{P}\in \mathcal{P}_{(m,\sigma)}} \mathbb{E}_\mathbb{P}[W(s,q(p)\Lambda)]$. That is, the worst-possible expected market share arises from the worst-possible utility and hence worst-possible expected waiting time. To make this more precise, notice that the maximin problem \eqref{mpp} can also be written as an optimization problem in terms of the equilibrium joining strategy $q(p)$; that is,
\begin{equation}\label{mppalt}
\begin{aligned}
\max_{q(p) \in[0,1] \cap[0, s\mu / \ulambda)} \min_{\P\in\cP_{(m,\sigma)}} \mathbb{E}_{\P}[q(p) \Lambda] &\cdot \mathbb{E}_{\P}[r-c W(s,q(p)\Lambda)] \\ &\equiv \max_{q(p) \in[0,1] \cap[0, s\mu / \ulambda)} \min_{\P\in\cP_{(m,\sigma)}}  \mathbb{E}_{\P}[q(p) m\left(r-c W(s,q(p)\Lambda)\right)],
\end{aligned}
\end{equation}
where the equivalence follows from the fact that the expected value $\E[\Lambda]=m$ is a known constant as it is part of the information contained in the ambiguity set. To see that \eqref{mppalt} is equivalent to \eqref{mpp}, we carefully go over the decision processes of the firm and the customers. First, the firm sets a price $p$, and then the customers settle on a joining strategy with expected utility $U=0$, while considering all possible distributions that possibly govern the market size $\Lambda$. That is, the customers solve the equation $r=p+\max_{\P\in\cP_{(m,\sigma)}} c\E_\P[W(s,q(p)\Lambda)]$, with the expected value operator appearing
here because of the customers’ internal calculations. In equilibrium, the utility equation can then be rewritten as $p=\min_{\P\in\cP_{(m,\sigma)}}\E_{\P}[r-cW(s,q(p)\Lambda)]$. Now, if the firm sets the price to $p$, the expected revenue will equal this price times the market share, $p\cdot\E_\P[\Lambda_e(p)]=p\cdot\E_\P[q(p)\Lambda]$, yielding \eqref{mppalt}.
In order to solve the utility equation, we must make additional assumptions regarding how customers experience waiting times in a mixed Poisson model. 

The next question then is a rather philosophical one: What to assume for the wait expected by an arbitrary arriving customer? One natural choice is $\mathbb{E}_{\mathbb{P}}[W(s,q\Lambda)]$; see, e.g., \cite{chen2017staffing,chen2020knowledge}. Here, we assume nature picks a realization $\Lambda=\lambda$, and customers from time 0 onward arrive according to a Poisson process with this rate. Customers remain unaware of the specific universe (i.e., event $\Lambda=\lambda$) they live in but hold a (Bayesian) prior belief about the probability of residing in one universe compared to another. Another option is to base the rational decision made upon arrival on the posterior distribution of $\Lambda$, the updated version of the prior distribution of $\Lambda$ that accounts for size bias. That is, the customer uses the conditional distribution of $\Lambda$ given the realization of its own arrival event. \citet{hassin2021strategic} call this phenomenon RASTA (Rate-biased ASTA), as a counterpart of PASTA. For any function $g(\cdot)$, the posterior expectation of $g(\Lambda)$ at arrival instants is $\E[\Lambda g(\Lambda)]/\E[\Lambda]$, and when $g$ is nondecreasing and convex, then
$\E[\Lambda g(\Lambda)]/\E[\Lambda] \geq \E[g(\Lambda)]  \geq g(\E[\Lambda])$. As an application, consider the M/M/$s$ queue. Then, the expected waiting time with a random arrival rate $q\Lambda$ equals 
$$
\overbar{W}(s,q):=\frac{\E[\Lambda W(s, q \Lambda)]}{\E[\Lambda]}=\frac{
\E[q \Lambda W(s, q \Lambda)]}{\E[q \Lambda]}=\frac{
\E[L(s, q \Lambda)]}{\E[q \Lambda]}.
$$
The utility of a customer who joins the queue when all the
other customers use strategy $q$ equals
$$U(q) = r-p- c \overbar{W}(s,q)$$
and the best response of an individual customer is to join if and only if $U(q) \geq  0$. 
When accounting for size bias, we thus need to solve
 \begin{equation}\label{optc1}
 \max_{\mathbb{P}\in \mathcal{P}} \overbar{W}(s,q)
 \end{equation}
to find the worst-case market share.
Alternatively, assuming (P)ASTA instead of RASTA, we have to consider 
\begin{equation}\label{optcx}
\max_{\mathbb{P}\in \mathcal{P}} \mathbb{E}_{\mathbb{P}_\Lambda}[ W(s,q\Lambda)].
\end{equation} 
It is easy to show that \eqref{optc1} is also solved by the distribution in Lemma~\ref{lemma:ec1}(ii), as it is known that $\lambda W(s,\lambda)$ is increasingly convex in $\lambda$ (see \citep{randhawa2016optimality}) and $\E[\Lambda]$ is a known constant since the mean $m$ is contained in the information set. For the RASTA arrival assumption, we assume the market consists of individuals that base their decision on the posterior distribution of $\Lambda$. Then, the worst-case joining probability $q=q^*$ solves the equation
$$
\frac{r-p}c = p_1  \frac{\lambda_1 W(s,q\lambda_1)}{m} + p_2  \frac{\lambda_2W(s,q\lambda_2)}{m},
$$
with $\lambda_1,\ \lambda_2$, $p_1$ and $p_2$ as defined in \textnormal{Theorem~\ref{thmlambda2}}.
For PASTA arrivals, it is assumed the market consists of individuals that base their decision on the prior distribution of $\Lambda$. Then, the worst-case joining probability is found by solving in $q$ the equation
$$
\frac{r-p}c = p_1  {W(s,q\lambda_1)} + p_2  W(s,q\lambda_2),
$$
with $\lambda_1,\ \lambda_2$, $p_1$ and $p_2$ again as defined in \textnormal{Theorem~\ref{thmlambda2}}.
For conciseness, let $\P^*$ denote this extremal distribution. We next determine the optimal price for PASTA arrivals. To accomplish this, we start by deriving an expression for the derivative of the expected waiting time with respect to $\lambda$, using the equations for $L(s,\rho)$ and $L'(s,\rho)$ from \cite{grassmann1983convexity}.
By substituting $\rho=\lambda/(s\mu)$ into the expression for $L'(s,\rho)$ and applying Little's law, the product rule, and the chain rule, we can determine 
$$
M(s,\lambda) := \frac{\partial}{\partial \lambda} W(s,\lambda)  = 
\frac{\partial}{\partial \lambda} \frac{L(s,\lambda)}{\lambda}.
$$

\begin{proposition}[Optimal joining probability, PASTA]\label{prop:optjoinPASTA}
Consider the M/M/$s$ queue with PASTA arrivals. Suppose that $c/\mu<r<c\E_{\P^*}[W(s,\Lambda)]$. Then the joining probability $q^*$ that maximizes the revenue for the firm, in the worst-case market $\Lambda\sim\P\in\cP_{(m,\sigma)}$, is the unique solution in $q$ of
\begin{equation}\label{eq:optjoinstrategyP}
r -c \frac{\E_{\P^*}[\Lambda M(s,q\Lambda)]}{\E_{\P^*}[\Lambda]}=0.
\end{equation}
\end{proposition}
\begin{proof}
We shall show that the solution to \eqref{eq:optjoinstrategyP} is the joining probability that maximizes
$$
\max_{q \in(0,1) \cap(0, s\mu / \ulambda)} \min_{\P\in\cP_{(m,\sigma)}}  \mathbb{E}_{\P}[q m\left(r-c W(s,q\Lambda)\right)].
$$
It suffices to consider $q\in(0,1) \cap(0, s\mu / \ulambda)$ as, due to the assumption, both $q=0$ and $q=1$ are not equilibrium strategies. It is fairly straightforward to see that the minimization problem is solved by $\P^*$, since the expected waiting time has the required third-order property. We next solve the maximization part. Taking the derivative with respect to $q$, we obtain the first-order condition
$$
rm - c \E_{\P^*}[\Lambda M(s,q\Lambda)]=0,
$$
which can be written as \eqref{eq:optjoinstrategyP}. Notice that we are allowed to interchange the differentiation and expectation operations as $\Lambda$ is bounded. Since $M(s,\cdot)$ is an increasing function of $\lambda$, as $W(s,\cdot)$ is (strictly) convex, this implies that \eqref{eq:optjoinstrategyP} has a unique solution that is the global maximizer, as the objective function is a concave function of $q$.
\end{proof}

We next determine the minimax price $p$ that induces the optimal joining strategy $q^*$.
\begin{proposition}[Optimal maximin price, PASTA]\label{prop:optpricePASTA}
For the M/M/$s$ queue with PASTA arrivals, the optimal price for the firm that solves the maximin problem \eqref{mpp} is
\begin{equation}\label{eq:optminimaxpriceP}
p^*=c \frac{\E_{\P^*}\left[\Lambda M\left(s,q^* \Lambda\right)-mW\left(s,q^* \Lambda\right)\right]}{\E_{\P^*}[\Lambda]}.
\end{equation}
\end{proposition}
\begin{proof}
It is required to show that the right-hand side of \eqref{eq:optminimaxpriceP} yields a solution to $q(p)=q^*$, or equivalently, $U(q^*)-p=0$. Indeed, plugging in the right-hand side of \eqref{eq:optminimaxpriceP}, we see that
$$
\begin{aligned}
U(q^*)-p &= r-c\max_{\P\in\cP_{(m,\sigma)}}\E[{W}(s,q^*\Lambda)]-c \frac{\E_{\P^*}\left[\Lambda M\left(s,q^* \Lambda\right)- m W\left(s,q^* \Lambda\right)\right]}{\E_{\P^*}[\Lambda]} \\
&=r-c\frac{\E_{\P^*}[\Lambda M(s,q\Lambda)]}{\E_{\P^*}[\Lambda]}=0,
\end{aligned}
$$
in which the final identity follows from Proposition~\ref{prop:optjoinPASTA}.
\end{proof}

\subsection{Numerical example}\label{sec:numerics}
We next present numerical results for the M/M/2 model with random market size $\Lambda\sim\P\in\cP_{\Lambda}(1.5,\sigma,0,2)$, $c=1$, $r=2$ and PASTA arrivals. Figure~\ref{fig:optqq} illustrates the equilibrium joining probabilities and revenues for various price values $p$ and different levels of variance $\sigma^2$. As expected, both the joining probabilities and revenues decrease with an increase in the dispersion of the market size. This observation aligns with the intuition that greater variance indicates more uncertainty about the market size and, as a result, greater uncertainty about the expected waiting time. Another noteworthy characteristic of the ambiguous model, as compared to the model with a known arrival rate, is that the equilibrium joining probability cannot exceed a certain level since, otherwise, the M/M/$s$ system can become unstable.

\begin{figure}[h!]
\begin{subfigure}{0.49\textwidth}
    \centering

    \includegraphics[width=\linewidth]{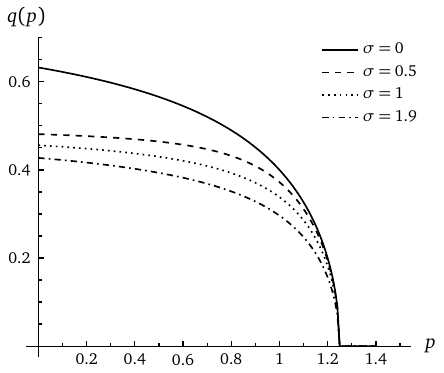}
\end{subfigure}
\begin{subfigure}{0.49\textwidth}
    \centering
    \includegraphics[width=\linewidth]{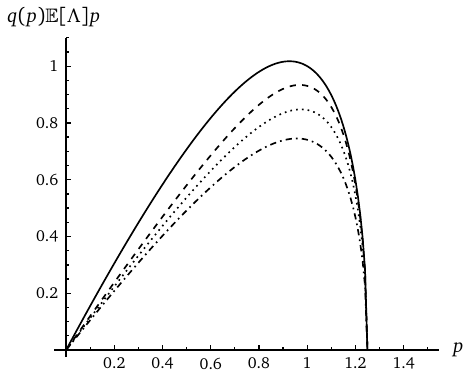}
\end{subfigure}
\caption{Equilibrium probability and revenue for different prices $p$ in the M/M/2 model with random market size $\Lambda\sim\P\in\cP_{\Lambda}(2.5,\sigma,0,4)$, $c=0.75$, $r=2$, and PASTA arrivals.}
\label{fig:optqq}
\end{figure}

In conclusion, we have shown that our second-order bounds can aid in deriving the optimal maximin price for a firm seeking to maximize revenue. This firm serves a rational customer base that makes decisions based on limited information about the overall market size. Interestingly, this adverse market follows a computationally tractable two-point distribution, thus allowing for the application of established techniques designed for rational queueing models with stochastic arrival rates.

\section{Conclusions and further research}\label{MMSsec6}
We have presented \new{tight} second-order bounds for the expected waiting time in an M/M/$s$ queueing system
\new{with an uncertain arrival rate.
We have derived these bounds 
by using the fact that the expected waiting time can be written as the expectation of a convex function with a convex derivative. This convex derivative, in conjunction with mean-variance information, gave rise to a specific semi-infinite linear program with an explicit solution in terms of the two-point worst-case distribution. }


\new{Looking forward, it is worthwhile to search for other information sets that have, just as mean-variance, explicit and perhaps even two-point worst-case distributions.
In Appendix~\ref{semiapp}, we consider semi-variance and show that the worst-case distribution is either a three-point or a two-point distribution with some support points that can only be determined numerically. Hence, while variance leads to a closed-form two-point distribution, semi-variance leads to a less explicit worst-case scenario. The proof in Appendix~\ref{semiapp} is based on solving the semi-infinite linear program with a primal-dual argument and again uses the convex derivative.  }
\new{Several recent papers have considered mean absolute deviation instead of variance, which leads for convex functions to an extremal three-point distribution that, as with variance, no longer depends on the objective function's precise form; see, for example, \cite{postek2018robust,van2022mad}.
}

One might also consider data-driven information sets, such as the Wasserstein ambiguity set, which are equipped to handle limited data, as estimates of means and dispersion measures might not always be statistically accurate enough. Additionally, it is insightful to investigate to which other stochastic systems analogous types of distributionally robust techniques to control for parameter uncertainty can be applied, a topic we briefly touched upon in Section~\ref{sec:furtherapplic}.

\bibliography{arXiv}
\bibliographystyle{apalike}


\begin{appendix}

\section{Proof of Lemma~\ref{lemma:ec1}}
\begin{proof}

From the derivations in \cite{grassmann1983convexity}, we know that the derivative of $L$ can be written as
\begin{equation}\label{eq:Lfirstder}
L'(\rho)=s+\left(\frac{s-L(\rho )+1}{\rho }+\frac{2}{1-\rho }\right) (L(\rho )-s \rho )
\end{equation}
Further differentiating both sides yields
\begin{equation}\label{eq:Lsecondder}
L''(\rho)=(L(\rho )-s \rho ) \left(-\frac{s-L(\rho )+1}{\rho ^2}-\frac{L'(\rho )}{\rho }+\frac{2}{(1-\rho )^2}\right)+\left(\frac{s-L(\rho )+1}{\rho }+\frac{2}{1-\rho }\right) \left(L'(\rho )-s\right)
\end{equation}
and
\begin{equation}\label{eq:Lthirdder}
\begin{aligned}
L'''(\rho)=\left(\frac{s-L(\rho )+1}{\rho }-\frac{2}{\rho -1}\right) L''(\rho )+2 \left(\frac{-s+L(\rho )-1}{\rho ^2}-\frac{L'(\rho )}{\rho }+\frac{2}{(\rho -1)^2}\right) \left(L'(\rho )-s\right) \\+  (L(\rho )-s \rho ) \left(\frac{2 (s-L(\rho )+1)}{\rho ^3}-\frac{L''(\rho )}{\rho }+\frac{2 L'(\rho )}{\rho ^2}-\frac{4}{(\rho -1)^3}\right).
\end{aligned}
\end{equation}
By Little's law, $W(\rho)=L(\rho)/(\rho s\mu)$. The third derivative of the expected waiting time is then given by
\begin{equation}\label{eq:thirdder}
    W'''(\rho)=\frac{\partial^3}{\partial \rho^3}\frac{L(\rho)}{\rho s\mu} =\frac{\rho  \left(6 L'(\rho )+\rho  \left(\rho  L'''(\rho )-3 L''(\rho )\right)\right)-6 L(\rho )}{s \mu  \rho ^4}.
\end{equation}
To prove the claim, we shall show that $W'''(\rho)\geq0$. Consecutively substituting  \eqref{eq:Lthirdder}, \eqref{eq:Lsecondder}, \eqref{eq:Lfirstder} and finally \eqref{eq:L0thder} into \eqref{eq:thirdder}, we obtain
\begin{equation}
\begin{aligned}
W'''(\rho)=\frac{C}{s \mu  (1-\rho)^4 \rho ^3} \Big(-s^2 (1-\rho )^4 &((7 C-6) \rho +3)+2 s (1-\rho)^2 (\rho  (2 C ((3 C-7) \rho +1)+9 \rho -4)+1)  \\  &-6 (C-2) ((C-2) C+2) \rho ^3+s^3 (1-\rho)^6\Big).
\end{aligned}
\end{equation}
Since
$
\frac{C}{s \mu  (1-\rho)^4 \rho ^3}
$
is nonnegative, it suffices to show that 
$$
\begin{aligned}
f(s,C):=\Big(-s^2 (1-\rho )^4 &\left((7 C-6) \rho +3 \right) + 2 s (1-\rho)^2 (\rho  (2 C ((3 C-7) \rho +1)+9 \rho -4)+1) \\ &-6 (C-2) ((C-2) C+2) \rho ^3+s^3 (1-\rho)^6 \Big) \geq 0.
\end{aligned}
$$
To simplify some of the terms of $f(s,C)$, it is convenient to work with bounds for $C$ in the remainder of the proof. We will use the simple bounds (i) $C\leq\rho$ and (ii) $C\leq1+\frac{s(1-\rho)^2}{2 \rho}-\frac{(1-\rho)}{2 \rho} \sqrt{4 s \rho+s^2(1-\rho)^2}$ (see, e.g., \citep{harel2010sharp}). It suffices to show for all $s\geq2$ that $f(s,C(\rho))\geq0,\ \forall\rho\in(0,1),$  since the result is already demonstrated for the single-server queue. We proceed by showing that $f(n,C)$ is nondecreasing for $2\leq n\leq s$. Then, to complete the proof, it remains to show that $f(2,C)\geq0$ for all $\rho\in(0,1)$.
Now notice that
$$
g(x) := \frac{1}{(1-\rho)^2}\frac{\partial}{\partial n}f(n,x) = a_2x^2-a_1 x+a_0, 
$$
where
$$
\begin{aligned}
    a_2 &= 12 \rho^2, \\
    a_1 &= 2\rho\left(7 n (1-\rho)^2 + 14 \rho  -2\right), \\
    a_0 &= 18 \rho^2-8 \rho +3 (1-\rho )^4 n^2-6 (1-\rho )^2 n+12 \rho  (1-\rho )^2 n+2. \\
\end{aligned}
$$
After some algebra, one sees that, for $n\geq2$, $a_0,a_1\geq0$, and clearly, $a_2=12\rho^2>0$. Since $a_0,a_1,a_2\geq0$, $g(x)$ has two positive roots. Denote the smaller root by $x^{-}$. Now, to show $g(C)\geq0$, we demonstrate that $g'(C)\leq g'(x^{-})$, which is sufficient as $g(x)$ is a convex quadratic function. We will instead prove this inequality for an upper bound on $C$:
$$
C\leq 1+\frac{n(1-\rho)^2}{2 \rho}-\frac{(1-\rho)}{2 \rho}\sqrt{4 n \rho+n^2(1-\rho)^2} =:\Bar{C}
$$
where the inequality follows from (ii) and $n\leq s$.
We next compute $g'(\bar{C})$ and $g'(x^{-})$, and show that $g'(x^{-}) - g'(\bar{C})$ is nonnegative. Demonstrating $g'(\bar{C})\leq g'(x^{-})$ is sufficient since $g$ is a quadratic decreasing function on $[0,\Bar{C}]$. Notice that
$$
\begin{aligned}
    g'(x^{-}) &= -2 \rho(1-\rho) \sqrt{ \left(n \left(52 \rho +13 (1-\rho)^2 n+44\right)-20\right)}\\
    g'(\bar{C}) &= -2 \rho(1-\rho )  \left((1-\rho) n+6 \sqrt{n \left(4 \rho +(1-\rho)^2 n\right)}-2\right).
\end{aligned}
$$
Since $g'(x^{-}),g'(\bar{C})<0$, it is sufficient to show $\frac{|g'(x^{-})|}{2\rho(1-\rho)}\leq\frac{|g(\bar{C})|}{2\rho(1-\rho)}$ by demonstrating nonnegativity of
$$
\begin{aligned}
\bar{g}(n,\rho)&:=\left(\frac{g(\bar{C})}{2\rho(1-\rho)}\right)^2-\left(\frac{g'(x^{-})}{2\rho(1-\rho)}\right)^2 \\&= \left(6 \sqrt{n \left(4 \rho +(1-\rho )^2 n\right)}+(1-\rho)n-2\right)^2-n \left(52 \rho +13 (1-\rho)^2 n+44\right)+20.
\end{aligned}
$$
After some tedious calculations, it follows from standard calculus that $\frac{\partial \bar{g}}{\partial n}\geq0$, for $n\geq2$. So, for fixed $\rho$, the auxiliary function $\bar{g}(\cdot,\rho)$ is minimized at $n=2$. Hence,
$$
\bar{g}(n,\rho) \geq \bar{g}(2,\rho) = 96 \rho^2 - 48\rho\sqrt{1+\rho^2} +24 \geq 17.5692 >0,
$$
for all $\rho\in(0,1)$. Here the second-to-last inequality follows from minimizing $96 \rho^2 - 48\rho\sqrt{1+\rho^2} +24 $.
Therefore, $g(C)\geq0$ or, equivalently, $\frac{\partial}{\partial n}f(n,C)\geq0$ for $2\leq n\leq s$.
In the remainder of the proof, it thus suffices to concentrate on $f(2,C)$. Define
$$
\begin{aligned}
h(x):=f(2,x)=\Big(-4(1-\rho )^4   &\left((7 x-6) \rho +3 \right) + 4 (1-\rho)^2 (\rho  (2 x ((3 x-7) \rho +1)+9 \rho -4)+1) \\ &-6 (x-2) ((x-2) x+2) \rho ^3+ 8(1-\rho)^6 \Big).   
\end{aligned}
$$
Observe that
$$
h''(x) = 12 \rho^2 (\rho (4 \rho-3 x-4)+4).
$$
From the well-known bound (i), it follows that $h''(x)\geq 0$ for $x\leq\rho$ since $h''(\rho)=12(\rho-2)^2\rho^2\geq0$. Further, it can be shown that
$$
h'(\rho) = 2 \rho (\rho^4 + 4 \rho^3 - 18 \rho^2 + 20 \rho - 10) \leq0,\ \forall\rho\in(0,1).
$$
Since $h''(x)\geq 0$, $h'(x)\leq0$ for all $x\leq\rho$. Hence, for $\rho\in(0,1)$,
$$
f(2,C)=h(C)\geq h(\rho) = -2 \rho^2 ((\rho - 2) \rho ((\rho - 2) \rho + 2) - 2)\geq0,
$$
in which the final inequality follows from some straightforward calculus. This completes the proof.
\end{proof}

\section{Semi-variance as information set}\label{semiapp}
We next consider a specific asymmetric dispersion measure that only measures dispersion above the mean.
Let $\cP_{(m,\semivar)}:=\mathcal{P}(m,\semivar,\llambda,\ulambda)$ be the set of all distributions with mean $m$, (upper) semivariance $\semivar:=\E[(X-\mu)^+]^2$, where $(x)^+=\max\{x,0\}$, and the support contained in the interval $[\llambda,\ulambda]$. Using primal-dual techniques, we establish the following result.

\begin{theorem}[M/M/$\boldsymbol{s}$ queue with semivariance information]\label{thmlambdasemi}
Consider an {M/M/$s$} queue with random arrival rate $\Lambda$ that follows a distribution $\mathbb{P}$ belonging to the ambiguity set  $\mathcal{P}(m,\semivar,\llambda,\ulambda)$. Suppose that $m\in(\llambda,\ulambda)$ and $\semivar\in(0,\frac{(\ulambda-m)^2(m-\llambda)}{(\ulambda-\llambda)})$. Then, the tight upper bound for the expected waiting time 
corresponds with the maximum value of $\mathbb{E}_{\mathbb{P}}[{W}(s,\Lambda)]$ that results from the following two solutions:
\begin{itemize}
\item[(i)] the expected waiting time with the expectation taken over a three-point distribution with 
$$
\mathbb{P}(\Lambda=\llambda )=p_1(x^*_0),\ \mathbb{P}(\Lambda=x^*_0)=p_2(x^*_0)  ,\ \mathbb{P}(\Lambda=\ulambda)=p_3(x^*_0),
$$
where $x^*_0\in[\mu,\ulambda)$ solves 
$$
\max_{x_0\in[\mu,\ulambda)} p_1(x_0) W(s,\llambda) + p_2(x_0)  W(s, x_0) + p_3(x_0) W(s,\ulambda)\  \textnormal{ s.t. }\  0\leq p_1,p_2,p_3\leq1;  \textnormal{ or }
$$

\item[(ii)] the expected waiting time with the expectation taken over the two-point distribution
$$
\begin{aligned}
\mathbb{P}(\Lambda=\llambda )&=\frac{\sqrt{4\semivar (m-\llambda)^2+\semivar^2}-\semivar }{2 (m-\llambda)^2},\ \mathbb{P}(\Lambda=x^*_0)=\frac{2 (m-\llambda)^2}{\sqrt{4\semivar  (m-\llambda)^2+\semivar^2 }+2 (m-\llambda)^2+\semivar }  ,\\ \textnormal{ where } x^*_0 &= \frac{2 m (m-\llambda) + \sqrt{4\semivar  (m-\llambda)^2+\semivar^2}+\semivar }{2 (m-\llambda)}.
\end{aligned}
$$

\end{itemize}
\end{theorem}

\begin{proof}
The primal problem for mean-(upper)semivariance information is given by
\begin{equation}\label{test5}
\begin{aligned}
&\max_{\P(x)\geq0} &  &\int_{\llambda}^{\ulambda} \phi(x){\rm d} \P(x)\\
&\text{s.t.} &      & \int_{\llambda}^{\ulambda} {\rm d}\P(x)=1,\  \int_{\llambda}^{\ulambda} x\,{\rm d}\P(x)=m,\ \int_{\llambda}^{\ulambda}((x-m)^+)^2 {\rm d}\P(x)=\semivar,   
\end{aligned}
\end{equation}
and admits the following dual:
\begin{equation}\label{test6}
\begin{aligned}
&\min_{\pi_0,\pi_1, \pi_2} &  &\pi_0 +\pi_1 m+\pi_2 \semivar\\
&\text{s.t.} &      & M_{\semivar}(x):=\pi_0 +\pi_1 x +\pi_2((x-m)^+)^2 \geq \phi(x), \ \forall x\in[\llambda,\ulambda].
\end{aligned}
\end{equation}

The objective function $\phi(\cdot)$ is increasingly convex since it represents $W(s,\cdot)$. Under the conditions imposed on the parameters of the ambiguity set, strong duality holds and the optimal values of the primal and dual problem coincide. In addition, since the common optimal value is finite, the dual optimal solution is attained \citep[Proposition~3.4]{shapiro2001duality}. Moreover, as both the objective function and the dual function $M_{\semivar}(\cdot)$ are continuous and the support $[\llambda,\ulambda]$ is compact, the optimal primal solution is also attained \citep[Corollary~3.1]{shapiro2001duality}. As a consequence, complementary slackness holds \cite[Proposition~2.1]{shapiro2001duality}. 
We next use the complementary slackness property and structural properties of the functions $\phi$ and $M_{\semivar}$ to determine which points should constitute the support of the extremal distribution. Since the primal problem has three constraints, we can restrict our search to a worst-case distribution with at most three support points \citep{rogosinski1958moments}. Since the conditions in the theorem impose $\semivar>0$, we can readily exclude the case of a one-point worst-case distribution.  Furthermore, we need both a point below and above the mean to satisfy the mean condition and to avoid the semivariance being equal to zero. For the support point below the mean, $\llambda$ is the only dual-feasible option by convexity of $\phi(\cdot)$, as $M_{\semivar}(\cdot)$ is linear for all $x\in[\llambda,m]$. We next consider two dual solutions that correspond with cases (i) and (ii) stated in the theorem, which assert a worst-case two- and three-point distribution, respectively.

\begin{figure}[h]
\begin{subfigure}{.49\linewidth}
\centering
\includegraphics[width=.75\linewidth]{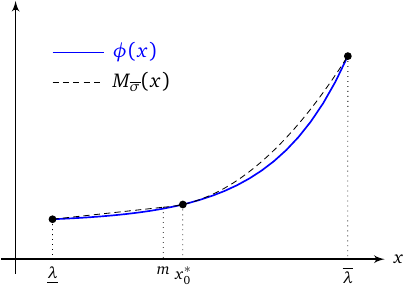}
\end{subfigure}
\begin{subfigure}{.49\linewidth}
\centering
\includegraphics[width=.75\linewidth]{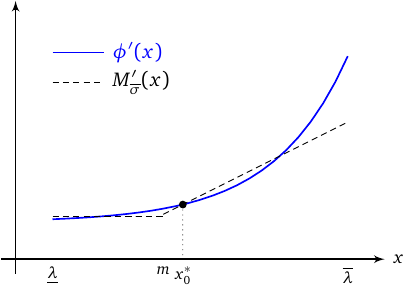}
\end{subfigure}
\caption{Objective function $\phi$, dual function $M_{\semivar}$ and their derivatives, where $M'_{\semivar}$ is interpreted as the right-derivative}
\label{figsemi}
\end{figure}

(i) Fixing the first support point to $\llambda$, we next seek the other two points using another complementary slackness argument. The dual function $M_{\semivar}$ can only be tangent to $\phi$ at one unique point $x_0$ on the interval $[m,\ulambda]$. This follows from linearity of $M_{\semivar}'(\cdot)$ and convexity of $\phi'(\cdot)$, as illustrated in Figure~\ref{figsemi}, which rules out the possibility of a second tangent point. To see this, assume that there exists a second tangent point $\tilde{x}_0$. To remain dual-feasible, the dual function has to satisfy $M_{\semivar}'(x)\leq\phi'(x)$ for $x\uparrow\tilde{x}_0$, and $M_{\semivar}'(x)\geq\phi'(x)$ for $x\downarrow \tilde{x}_0$. Therefore, $M_{\semivar}'(\cdot)$ has to intersect $\phi'(\cdot)$ from below at $\tilde{x}_0$, but since this already occurred at $x_0^*$, this cannot happen a second time as otherwise $M_{\semivar}'(\cdot)$ has to be nonlinear or $\phi'(\cdot)$ nonconvex. Hence, we arrive at a contradiction. The dual function can only coincide with $\phi(\cdot)$ one more time, at the upper bound of the support, $x=\ulambda$. Therefore, if the worst case is given by a three-point distribution, then it has to admit this specific form as a result of complementary slackness. The probabilities, as functions of $x_0$, follow from solving the moment constraints in \eqref{test5}. As $x_0$ is yet to be determined, the maximum value follows from solving the univariate optimization problem stated in the claim. There always exists a feasible solution to this optimization problem. Specifically, letting $x_0=m$, we obtain the three-point distribution
$$
p_1=\frac{\semivar}{(\ulambda-m)(m-\llambda)},\ p_2= 1-\frac{\semivar(\ulambda-\llambda)}{(m-\llambda) (\ulambda-m)^2},\ p_3=\frac{\semivar }{(\ulambda-m)^2},
$$
of which is easily verified, using the conditions in the claim, that it is primal feasible for all possible parameter combinations.

(ii) We next consider the two-point solution. If the dual function only coincides with $\phi$ at $\llambda$ and some $x_0$, then the resulting two-point distribution can be directly derived from the moment conditions
$$
p_1  + p_2  = 1,\ p_1 \llambda + p_2 x_0^* = m,\  p_2 (x_0^*-m)^2 = \semivar,
$$
which is a system of three equations with three unknowns with a unique solution, as stated in the claim. Further, from the assumptions in the claim, it follows that this distribution is always feasible in the primal. 
Taking the maximum expected value of the assertions (i) and (ii) completes the proof of the claim.
\end{proof}


\end{appendix}

\end{document}